\documentclass{aims} % Use the aims.cls file to compile your paper
\usepackage{amsmath}
\usepackage{paralist}
\usepackage[misc]{ifsym}
\usepackage{epsfig} %For pictures: screened artwork should be set up with an 85 or 100 line screen
\usepackage{epstopdf} %This is to transfer .eps figure to .pdf figure; please compile your article using PDFLaTex or PDFTeXify.
\usepackage[colorlinks=true]{hyperref}
\hypersetup{urlcolor=blue, citecolor=red}
\allowdisplaybreaks

\def\currentvolume{X}
 \def\currentissue{X}
  \def\currentyear{200X}
   \def\currentmonth{XX}
    \def\ppages{X-XX}
     \def\DOI{10.3934/xx.xxxxxxx}

\usepackage{bm}
\usepackage{tikz}

\def\div#1{{\,\nabla \cdot #1}}
\def\d#1{\overline{#1}}
\def\x{{\bf x}}
\def\o{{\bf 1}}
\def\M{{\bf M}}
\def\A{{\bf A}}
\def\B{{\bf B}}

\def\Q{{\bf Q}}
\def\X{{\bf X}}
\def\Y{{\bf Y}}
\def\Z{{\bf Z}}
\def\b{{\bf b}}
\def\x{{\bf x}}

\def\f{{\bf f}}
\def\m{{\bf m}}
\def\q{{\bf q}}
\def\bphi{{\bm \phi}}
\def\dr{{\rm d}}

\newtheorem{theorem}{Theorem}[section]

\newtheorem{lemma}[theorem]{Lemma}
\newtheorem{proposition}[theorem]{Proposition}

\theoremstyle{definition}

\newtheorem{example}{Example}

\title[Measure preservation for LV tree-systems]
{Measure preservation and integrals for Lotka--Volterra tree-systems
and their Kahan discretisation} % Only the first word and proper nouns should be capitalized.

\author[P.H. van der Kamp$^*$, R.I. McLachlan, D.I. McLaren and G.R.W. Quispel]{}


\thanks{$^*$Corresponding author: Peter H. van der Kamp}

\begin{document}
\maketitle

% Enter the first author's name and email address; email addresses are required for each author.
% Use footnote notations to indicate address and affiliations with commas between numbers if more than one address applies; see below for examples.
\centerline{\scshape
Peter H. van der Kamp$^{{\href{mailto:P.VanDerKamp@latrobe.edu.au}{\textrm{\Letter}}}*1}$,
Robert I. McLachlan$^{{\href{mailto:R.McLachlan@massey.ac.nz}{\textrm{\Letter}}}*2}$,
David I. McLaren$^{{\href{mailto:D.McLaren@latrobe.edu.au}{\textrm{\Letter}}}*1}$,
and G. R. W. Quispel$^{{\href{mailto:R.Quispel@latrobe.edu.au}{\textrm{\Letter}}}1}$}

\medskip

{\footnotesize
% Enter the full affiliation and country name:
% Do not insert commas or periods at the end of lines.
 \centerline{$^1$La Trobe University, Australia}
} % Do not forget to end {\footnotesize with the sign }

\medskip

{\footnotesize
 % Enter the full affiliation and country name:
 \centerline{$^2$Massey University, New Zealand}
}

\bigskip

% The name of the handling editor will be entered by AIMS production staff.
% "Communicated by Handling Editor" is not needed for special issue.
 \centerline{(Communicated by Handling Editor)}

%%%%%%%%%%%%%%%%%%%%%%%%%%%%%%%%%%%%%%%%%%%%%%%%%%%%%%%
%             5. ABSTRACT
%%%%%%%%%%%%%%%%%%%%%%%%%%%%%%%%%%%%%%%%%%%%%%%%%%%%%%%

\begin{abstract}
We show that any Lotka--Volterra tree-system associated with an $n$-vertex tree, as introduced in Quispel et al., J. Phys. A 56 (2023) 315201, preserves a rational measure. We also prove that the Kahan discretisation of these tree-systems factorises and preserves the same measure. As a consequence, for the Kahan maps of Lotka--Volterra systems related to the subclass of tree-systems corresponding to graphs with more than one $n$-vertex subtree, we are able to construct rational integrals.
\end{abstract}

%%%%%%%%%%%%%%%%%%%%%%%%%%%%%%%%%%%%%%%%%%%%%%%%%%%%%%
%                   6. BODY
%%%%%%%%%%%%%%%%%%%%%%%%%%%%%%%%%%%%%%%%%%%%%%%%%%%%%%

% Only the first word and proper nouns of section titles should be capitalized.
% The title of section 1:
\section{Introduction}
An (autonomous) $n$-dimensional Lotka--Volterra (LV) system is a system of the form
\begin{equation}\label{eq:lv}
\dot{x_i}=x_i(b_i+\sum_{j=1}^n A_{i,j}x_j), \qquad i=1,\ldots,n,
\end{equation}
where the vector $\b$ and the matrix $\A$ do not depend on $\x(t)$ or $t$.

A polynomial $P(\x)$ is called a Darboux polynomial (DP) \cite{ce2,ga,go} for an ODE
\begin{equation} \label{ODE}
\dot{\x}=\f(\x)
\end{equation}
if there is a function $C(\x)$, called the cofactor of $P(\x)$, such that
\begin{equation} \label{ddp}
\dot{P}(\x)=C(\x)P(\x).
\end{equation}
More generally, if \eqref{ddp} holds but $P(\x)$ is not a polynomial, it is called a Darboux function. Each $n$-dimensional LV-system has $n$ DPs, namely the coordinates, $x_i$, themselves. Moreover, LV-systems are normal forms for quadratic ODEs with $n$ linearly independent linear DPs; by a linear transformation (introducing the DPs as new variables) such a system can be written in LV form.

Some LV-systems have additional Darboux polynomials. We recall a key lemma.
\begin{lemma}[\cite{KQM}]
The LV-system \eqref{eq:lv} has Darboux polynomial 
$
P_{i,k} := \alpha x_i + \beta x_k$ with $\alpha\beta\ne 0$ if and only if, for some constant $b$ and
for all $j\not\in\{i,k\}$,
\begin{align}
A_{i,j} &= A_{k,j} \label{eq:Aij}\\
b_i &= b_k = b \nonumber \\
\alpha(A_{k,k}-A_{i,k}) & = \beta(A_{k,i}-A_{i,i}) \nonumber
\end{align}
and $(A_{k,k}-A_{i,k})(A_{k,i}-A_{i,i})\ne 0$.
\end{lemma}

We will take such a DP to be
\begin{equation*} \label{eq:Pik}
P_{i,k} = (A_{k,i}-A_{i,i})x_i + (A_{k,k}-A_{i,k}) x_k.
\end{equation*}
Its cofactor is given by $\sum_j B_{j} x_j$, where
\begin{equation} \label{Bj}
B_j=\begin{cases} A_{j,j} & \text{if } j\in\{i,k\}\\
A_{i,j}=A_{k,j} & \text{if } j\not\in\{i,k\},\end{cases}
\end{equation}
cf. \cite[Definition 13]{KQM}.

If \eqref{eq:Aij} holds for several pairs $(i,k)$, the associated LV-system has several additional DPs. In this paper, we consider LV-systems (\ref{eq:lv}) with $\b={\bf 0}$\footnote{For a discussion on the relative importance of homogeneous vs inhomogeneous polynomial systems cf. e.g. \cite{Collins}.} and we view the pairs $(i,k)$ such that \eqref{eq:Aij} holds as edges of a graph with $n$ vertices. We shall see that the structure of the graph determines properties of the associated LV-system and of a certain birational map associated with it.

The case that the graph is a tree on $n$ vertices was considered in \cite{QTMK,KQM}, where 
$(3n-2)$-parameter families of homogeneous $n$-dimensional LV-systems, in one-to-one correspondence with trees on $n$ vertices, were shown to be superintegrable. These families will be referred to as tree-systems, or $T$-systems if it is clear that $T$ is a tree.\footnote{Tree-systems are not to be confused with $T$-systems, where $T$ stands for transfer matrix (or Toda or Tau) \cite[footnote 4]{KNS}.} To each edge of the tree $T$ corresponds a DP for the LV-system, and, by using the $n$ given DPs, $x_i$, one can then construct $n-1$ integrals. We illustrate the construction with an example, the bushy tree on 4 vertices, in section 2.

An ODE \eqref{ODE} is measure-preserving with measure
\[
\frac{\dr x_1\dr x_2\cdots\dr x_n}{d(\x)}
\]
if $d(\x)$, the density, is a DP with cofactor equal to the divergence of $\f(\x)$, i.e., it satisfies
\begin{equation*} %\label{dom}
\dot{d}(\x)=\left(\div \f(\x)\right)\ d(\x) .
\end{equation*}
We show, in section 2, that Lotka--Volterra $T$-systems are measure-preserving, with reciprocal density
\begin{equation} \label{dd}
d(\x) = \prod_{i=1}^n x_i^{2-m_i} \prod_{j=1}^{n-1} P_j,
\end{equation}
where, cf. \cite[Equation (6)]{KQM},
\begin{equation} \label{Pi}
\begin{split}
P_j&:=P_{u_j,v_j}\\
&=(A_{v_i,u_i}-A_{u_i,u_i})x_{u_i} + (A_{v_i,v_i}-A_{u_i,v_i}) x_{v_i}\\
&=(c_i-a_{u_i})x_{u_i} + (a_{v_i}-b_i) x_{v_i}
\end{split}
\end{equation}
is the DP obtained from the $j$th edge, $e_j=(u_j,v_j)$, of the tree $T$, and $m_i$ is the number of edges connected to the vertex $i\in T$.

In section 3 we consider the Kahan discretisation of tree-systems. The Kahan discretisation \cite{ce6,ka2,ki,ko,pe7,pe9} with step size $h$ of a homogeneous quadratic ODE
\[
\dot{x}_i=\sum_{j,k} c^i_{j,k}x_jx_k
\]
is the birational map $\x \mapsto \x'$ implicitly given (or defined) as follows by
\[
\frac{x'_i-x_i}{h}=\sum_{j,k} c^i_{j,k}\frac{x'_jx_k+x_jx'_k}{2}.
\]
We show that the Kahan discretisation of a tree-system is explicitly given by
\begin{equation*}%& \label{imp}
x'_i=x_i\frac{\prod_{j\neq i} K_{i,j}}{|\M|}
\end{equation*}
with
\begin{equation*}%\label{iKij}
K_{i,j}=1-\frac{h}{2}\left((\A.\x)_j+(A_{j,j}-A_{i,j})x_j\right),
\end{equation*}
and
\begin{equation*} %\label{M}
\M=\d{\o}-\frac{h}{2}\left(\d{\x}.\A+\d{\A.\d{\x}.\o}\right).
\end{equation*}
where $\d{\x}$ denotes the diagonal matrix with entries $\d{\x}_{ii}=x_i$.

A {\em discrete  Darboux polynomial} for a rational map $\x\mapsto \x' :=\bphi(\x)$, $\x\in\mathbb{R}^n$, is a polynomial $P\colon\mathbb{R}^n\to \mathbb{R}$ such that there exists a rational function $C\colon\mathbb{R}^n\to\mathbb{R}$ (again called the cofactor of $P$) whose denominator does not have any common factors with $P$, such that $P' = C P$ where $P':=P\circ\bphi$. 
If $P_1,\dots,P_k$ are Darboux polynomials with cofactors $C_1,\dots,C_k$, respectively, then $P := \prod P_i^{\alpha_i}$ obeys
$P'=CP$ where $C=\prod C_i^{\alpha_i}$. (If the $\alpha_i$ are not nonnegative integers, $P$ is  a Darboux function rather than a Darboux polynomial.) If, in addition, $C=1/\det D\bphi$, then $\frac{1}{P}\dr x_1\dots \dr x_n$ is an invariant measure of $\bphi$, while if $C=1$ then $P$ is a first integral of $\bphi$ \cite{ce1,MMQ}.

Linear Darboux polynomials are preserved under Kahan discretisation \cite[Theorem 1]{ce2}. We show that the cofactors of the preserved discrete Darboux polynomials, $P_i$, of the Kahan-discretised tree-systems are given by
\[
L_i\frac{\prod_{j\neq u_i,v_i} K_{u,j}}{|\M|}
\]
where
\[
L_{i}=1-\frac{h}{2}\left((\A.\x)_{u_i}-(A_{u_i,u_i}-A_{v_i,u_i})x_{u_i}\right).
\]
We prove that the Jacobian of the Kahan map of a tree-system is given by
\begin{equation} \label{jj}
J=\frac{\left(\prod_{i=1}^{n-1} L_i\right)\left(\prod_{i=1}^n \left(\prod_{j=1}^n K_{i,j}\right)/K_{i,i} \right)}{|\M|^{n+1}},
\end{equation}
and we prove that the expression $d(\x)$, given by \eqref{dd}, is a rational Darboux function with cofactor $J$, given by \eqref{jj}.
Thus, Kahan-discretised tree-systems are measure-preserving with density $(d(\x))^{-1}$.

In the final section, we consider $n$-dimensional LV-systems related to graphs $G$ that contain more than one subgraph which is a tree on $n$ vertices. The Kahan maps related to these so-called  $G$-systems preserve more than one measure and this enables us to find integrals for these maps. We classify distinct classes of $G$-systems and explicitly provide all distinct graphs on $4,5$ and $6$ vertices. We show that if $G$ contains a cycle of length $\ell$, the Kahan map of the $G$-system has at least $\ell-2$ integrals.

\section{Measure preservation for tree-systems}
We start this section by illustrating the construction of tree-systems with an example, whilst referring to where such systems were introduced. We then prove that all tree-systems are measure-preserving and provide an explicit expression for a density of the measure. This particular density will be preserved under Kahan-discretisation, which is the topic of the next section.

For any tree $T$ on $n$ vertices, one can associate a homogeneous Lotka--Volterra system, i.e. a system of the form
\begin{equation} \label{Tsy}
\dot x_i =  x_i \sum_{j=1}^n A_{i,j} x_j,\qquad i=1,\ldots,n
\end{equation}
with $3n-2$ free parameters \cite{QTMK,KQM}. The $n\times n$ matrix $\A$ is the adjacency matrix of the associated weighted complete digraph of $T$, cf. \cite[Definition 3]{KQM}.
\begin{figure}[h]
\begin{center}
\scalebox{.75}{
\begin{tikzpicture}
\node[shape=circle,draw=black,line width=0.5mm] (1) at (2,-1/2) {1};
	\node[shape=circle,draw=black,line width=0.5mm] (2) at (2,2) {2};
	\node[shape=circle,draw=black,line width=0.5mm] (3) at (0,4) {3};
	\node[shape=circle,draw=black,line width=0.5mm] (4) at (4,4) {4};
	\path [-,line width=0.5mm, ] (1) edge node[left] {$1$} (2);
	\path [-,line width=0.5mm, ] (2) edge node[below left] {$2$} (3);
	\path [-,line width=0.5mm, ] (2) edge node[below right] {$3$} (4);
\end{tikzpicture}
}
\caption{\label{bt4} The bushy tree on 4 vertices.}
\end{center}
\end{figure}
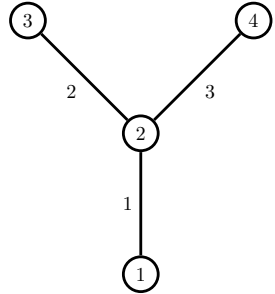

For the tree shown in Figure \ref{bt4} the matrix $\A$, with $3\times 4-2=10$ free parameters, is
\begin{equation} \label{Abt4}
\A=\begin{pmatrix}
a_{{1}}&b_{{1}}&b_{{2}}&b_{{3}}\\
c_{{1}}&a_{{2}}&b_{{2}}&b_{{3}}\\
c_{{1}}&c_{{2}}&a_{{3}}&b_{{3}}\\
c_{{1}}&c_{{3}}&b_{{2}}&a_{{4}}
\end{pmatrix}.
\end{equation}
We note that the cofactor of the DP $x_i$ is given by $(\A\x)_i$. The matrix $\A$ has the property that for each pair of rows $e_i=(u_i,v_i)\in\{(1,2),(2,3),(2,4)\}$ (that is, for each edge of $T$) we have $A_{u_i,k}=A_{v_i,k}$ for all $k\not\in e_i$. This property gives rise to $n-1$ additional Darboux polynomials of the form \eqref{Pi},
\begin{equation} \label{P123}
\begin{split}
P_1&=\left( c_{{1}}-a_{{1}} \right) x_{{1}}+ \left( a_{{2}}-b_{{1}} \right) x_{{2}}\\
P_2&=\left( c_{{2}}-a_{{2}} \right) x_{{2}}+ \left( a_{{3}}-b_{{2}} \right) x_{{3}}\\
P_3&=\left( c_{{3}}-a_{{2}} \right) x_{{2}}+ \left( a_{{4}}-b_{{3}} \right) x_{{4}}.
\end{split}
\end{equation}
The cofactors of $P_l$ is given by $(\B\x)_l$ where
\[
\B=\begin{pmatrix}
a_{{1}}&a_{{2}}&b_{{2}}&b_{{3}} \\
c_{{1}}&a_{{2}}&a_{{3}}&b_{{3}} \\
c_{{1}}&a_{{2}}&b_{{2}}&a_{{4}}
\end{pmatrix}.
\]
If the $l$th edge is $e_l=(i,k)$ then the element of $\B$ in row $l$ and column $j$ is given by \eqref{Bj}, cf. \cite[Definition 13]{KQM}. 
Using the rather general method \cite[section 2]{QTMK}, the additional DPs
give rise to $n-1$ integrals \cite[Equation (7)]{KQM}, for $i=1,2,3$:
\begin{equation*} %\label{ei1}
I_i=P_i^{|\A|}\prod_{k=1}^n x_k^{Z_{i,k}}, \qquad \Z=-\B\A^{-1}|\A|,
\end{equation*}
cf. \cite[section 4.3.2]{QTMK} for explicit expressions.

\begin{proposition} \label{meas}
Let $T$ be a tree on $n$ vertices, and let $m_i$ be the degree of (number of edges which meet at) vertex $i$. The Lotka--Volterra $T$-system \eqref{Tsy} is measure preserving with density
\begin{equation} \label{d}
d = \prod_{i=1}^n x_i^{2-m_i} \prod_{j=1}^{n-1} P_j,
\end{equation}
where $P_j$ is the DP associated with edge $e_j$, given by \eqref{Pi}.
\end{proposition}
\begin{proof}
A product of DPs $d=\sum_i p_i^{q_i}$ is a Darboux function, as $\dot{d}=(\sum_i q_iC_i) d$, where $C_i$ is the cofactor of $p_i$. Let
\[
p_i=\begin{cases} x_i &i=1,\ldots,n\\ P_{i-n} &i=n+1,\ldots,2n-1,\end{cases} \quad
q_i=\begin{cases} 2-m_i &i=1,\ldots,n\\ 1 &i=n+1,\ldots,2n-1.\end{cases}
\]
Then the cofactors are
\[
C_i=\begin{cases} \sum_{j=1}^n A_{i,j}x_j & i=1,\ldots,n \\ \sum_{j=1}^n B_{i-n,j}x_j & i=n+1,\ldots,2n-1, \end{cases}
\]
where the matrix $\B$ contains the coefficients of the cofactors of the additional DPs $P_j$, see \cite[Definition 13]{KQM}.
We have $\dot{x}_i=f_i=x_iC_i$ and $\div{\f}=\Big(\sum_{i=1}^n C_i+A_{i,i}x_i\Big)$. Hence
\begin{align}
\dot{d}/d-\div{\f}&=\sum_{i=1}^{2n-1} q_i C_i - \Big(\sum_{i=1}^n C_{i}+A_{i,i}x_i\Big)\notag\\
&=\sum_{i=1}^{n} (1-m_i)C_i + \sum_{i=1}^{n-1} C_{n+i} - \sum_{i=1}^n a_{i}x_i. \label{etr}
\end{align}
We will show that the coefficient of $x_p$ in the linear combination (\ref{etr}) vanishes for arbitrary $p\in\{1,\ldots,n\}$, i.e., that
\begin{equation} \label{lll}
\sum_{i=1}^{n} (1-m_i)A_{i,p} + \sum_{i=1}^{n-1} B_{i,p} - a_p=0.
\end{equation}
Recall that the edges of $T$ are given by $e_i=(u_i,v_i)$ for $i=1,\ldots,n-1$. For $p\in\{1,\ldots,n\}$, let $J^p,K^p$ be sets of indices such that
\[
j\in J^p \Leftrightarrow v_j=p,\qquad k\in K^p \Leftrightarrow u_k=p.
\]
Then the disjoint union $I^p=J^p\cup K^p$ has $m_p$ elements. We think of the tree $T$ as a collection of $m_p$ trees, $T^p_i$ ($i\in I^p$), connected at the vertex $p$. We define $z(p,i)$ to be the number of edges contained in $T^p_i$, so that, for each $p$, $\sum_{i\in I^p} z(p,i)=n-1$, which equals the number of edges in $T$.

For any tree $T$ on $n$ vertices, if $m_i$ is the number of edges at vertex $i$ and $e=n-1$ is the number of edges, then
\[
\sum_i m_i= 2e \implies \sum_i (1-m_i)= n - 2 e = 1 - e.
\]
Now, consider the first sum in \eqref{lll}. We break up the sum into $m_p$ sums plus a term.
\begin{align}
\sum_{i=1}^{n} (1-m_i)A_{i,p} &= \Big( \sum_{l=1}^{m_p} \sum_{i\in T^p_l} (1-m_i) A_{i,p} \Big) + (1-m_p)A_{p,p} \notag \\
&= \Big( \sum_{j\in J^p} (1-z(p,j))b_j \Big) + \Big( \sum_{k\in K^p} (1-z(p,k))c_k \Big) + (1-m_p) a_{p}, \label{AA}
\end{align}
as for each vertex $i\neq p \in T^p_j, j\in J^p \implies A_{i,p}=b_j$ and for each vertex $i\neq p \in T^p_k, k\in K^p \implies A_{i,p}=c_k$ (note $p\in T^p_i$ does not contribute to the sum as in $T^p_i$ only 1 edge meets in $p$). Next, consider the second sum in \eqref{lll}. Note that here the index $i$ runs over the edges, not the vertices, of $T$. We have
\begin{align}
\sum_{i=1}^{n-1} B_{i,p} &= \Big( \sum_{l=1}^{m_p} \sum_{e_i\in T^p_l} B_{i,p} \Big) \notag \\
&=\Big( \sum_{j\in J^p} (z(p,j)-1)b_j \Big) + \Big( \sum_{k\in K^p} (z(p,k)-1)c_k \Big) + m_p a_{p}, \label{BB}
\end{align}
as $i\in I^p \implies B_{i,p}=a_p$ and $i\not\in I^p, e_i=(v,w), B_{i,p}=A_{v,p}=A_{w,p}=b_j$ or $c_k$, depending on whether there is a $q\in T^p_l$ such that $(q,p)=e_j$ or $(p,q)=e_k$. By substitution of \eqref{AA} and \eqref{BB} into \eqref{lll} the result follows.
\end{proof}
We note that, due to the existence of many integrals, the ODE preserves many other measures. The measure introduced in Proposition \ref{meas} is the one that is preserved by Kahan discretisation.

\begin{example}
The 4-dimensional tree-system given by equation \eqref{Tsy} with \eqref{Abt4}, which is connected to the bushy tree displayed in Figure \ref{bt4}, has divergence
\[
\div{\f}=\left( 2\,a_{{1}}+3\,c_{{1}} \right) x_{{1}}+ \left( 2\,a_{{2}}+b_{{1
}}+c_{{2}}+c_{{3}} \right) x_{{2}}+ \left( 2\,a_{{3}}+3\,b_{{2}}
 \right) x_{{3}}+ \left( 2\,a_{{4}}+3\,b_{{3}} \right) x_{{4}}.
\]
The vector whose $i$-th component equals the degree of vertex $i$ is ${\bf m}=(1,3,1,1)$, so that ${\bf 2}-{\bf m}=(1,-1,1,1)$. According to Proposition \ref{meas} the density
\[
d=x_1(x_2)^{-1}x_3x_4P_1P_2P_3,
\]
with $P_1,P_2,P_3$ given by \eqref{P123}, is a rational Darboux function with cofactor $\div{\f}$. This can be verified by differentiation, or, alternatively, as follows. We write equation \eqref{lll} as ${\bf K}\x=0$, where
\[
{\bf K}=(\o_n-\m)\cdot \A+\o_{n-1}\cdot \B - {\bf a}.
\]
Then, in our case, we only have to compute
\begin{align*}
{\bf K}&=\begin{pmatrix}
0 & -2 & 0 & 0
\end{pmatrix}
\begin{pmatrix}
a_{{1}}&b_{{1}}&b_{{2}}&b_{{3}}\\
c_{{1}}&a_{{2}}&b_{{2}}&b_{{3}}\\
c_{{1}}&c_{{2}}&a_{{3}}&b_{{3}}\\
c_{{1}}&c_{{3}}&b_{{2}}&a_{{4}}
\end{pmatrix}
+
\begin{pmatrix}
1 & 1 & 1
\end{pmatrix}
\begin{pmatrix}
a_{{1}}&a_{{2}}&b_{{2}}&b_{{3}}\\
c_{{1}}&a_{{2}}&a_{{3}}&b_{{3}}\\
c_{{1}}&a_{{2}}&b_{{2}}&a_{{4}}
\end{pmatrix}\\
&\ \ \ 
-
\begin{pmatrix}
a_1 & a_2 & a_3 & a_4
\end{pmatrix}\\
&=-2\begin{pmatrix}
c_1&a_2&b_2&b_3
\end{pmatrix}
+
\begin{pmatrix}
a_1+2c_1&3a_2&a_3+2b_2&a_4+2b_3
\end{pmatrix}\\
&\ \ \ -
\begin{pmatrix}
a_1 & a_2 & a_3 & a_4
\end{pmatrix}\\
&={\bf 0}.
\end{align*}
\end{example}

\section{Measure preservation for the Kahan map of tree-systems}
The components of Kahan discretisation of tree-systems factorise into linear functions. These functions also appear in the cofactors of the DPs \eqref{Pi}. We establish an explicit expression for the Jacobian determinant of the Kahan map and show that the reciprocal density \eqref{d} is a rational Darboux function of the Kahan map which has the Jacobian determinant as its cofactor.

Let $\d{\x}$ be the diagonal matrix with entries $\d{\x}_{ii}=x_i$. The Kahan discretisation of \eqref{Tsy}, $\x\mapsto\x^\prime$,  satisfies $\M\x^\prime=\x$, where
\begin{equation*} %\label{M}
\M=\d{\o}-\frac{h}{2}\left(\d{\x}.\A+\d{\A.\d{\x}.\o}\right).
\end{equation*}
\begin{proposition} \label{expl}
The Kahan discretisation is explicitly given by
\begin{equation} \label{mp}
x_i^\prime=x_i\frac{\prod_{j\neq i} K_{i,j}}{|\M|}
\end{equation}
where
\begin{equation}\label{Kij}
K_{i,j}=1-\frac{h}{2}\left((\A.\x)_j+(A_{j,j}-A_{i,j})x_j\right).
\end{equation}
\end{proposition}
\begin{proof}
Let $\M^{(i)}$ be the matrix obtained from $\M$ by replacing the $i$th column by $\x$. By Cramer's rule we need to show
\[
|\M^{(i)}|=x_i\prod_{j\neq i} K_{i,j}.
\]
The off-diagonal entries in $\M$ are linear in $h$. The diagonal entries are affine, of the form $1+h(\cdots)$. Therefore we have $|\M^{(i)}|=x_i+h(\cdots)$. Hence, if
\[
|\M^{(i)}|=cx_i\prod_{j\neq i} K_{i,j},
\] then $c=1$, and it suffices to show that $x_i$ and $K_{i,j}$ ($j\neq i$) are divisors of $|\M^{(i)}|$. It follows that $x_i$ is a divisor by expanding $|\M^{(i)}|$ in the $i$th row. We prove that $K_{i,j}$ ($j\neq i$) is a divisor by establishing
\[
K_{i,j}=0\implies |\M^{(i)}|=0.
\]
Let ${\bf k}=k_1,k_2,\ldots,k_m$ be the path from $k_1=i$ to $k_m=j$, i.e., for all $l$ we have that either $(k_l,k_{l+1})$ or $(k_{l+1},k_l)$ is an edge in $T$. Let us create a matrix $\M^{[i]}$ by dividing the $j$th row of $\M^{(i)}$ by $x_j$. The $i$th column of $\M^{[i]}$ is ${\bf 1}$, and the other elements, apart from the diagonal ones, are $\M^{[i]}_{k,l}=-\frac{h}{2}A_{k,l}$. Modulo $K_{i,j}$ we have
\[
M^{[i]}_{j,j}\mid_{K_{i,j}=0}=\frac{1-\frac{h}{2}\big((\A.\x)_j+A_{j,j}x_j\big)}{x_j}\mid_{K_{i,j}=0}
        =-\frac{h}{2}A_{i,j}.
\]
From this, and from \cite[Definition 3]{KQM}, as ${\bf k}$ is the path from $i$ to $j$, it follows that
\begin{equation}\label{uf}
M^{[i]}_{i,j}=M^{[i]}_{k_l,j} \text{ for all } l=1,\ldots,m.
\end{equation}
Consider the $k_{m-1}$st and the $k_m$th row of $\M^{[i]}$. As $(k_m,k_{m-1})$ or $(k_{m-1},k_m)$ is an edge in $T$, due to \cite[Corollary 8]{KQM} and the fact that the $i$th column of $\M^{[i]}$ is ${\bf 1}$, we have
\begin{equation}\label{uf2}
M^{[i]}_{k_{m-1},l}=M^{[i]}_{k_{m},l} \text{ for all } l \neq k_{m-1},k_m.
\end{equation}
Because of \eqref{uf}, equation \eqref{uf2} also holds for $l=k_m$, and thus the rows differ only in the $k_{m-1}$st column. We can now add a (non-zero) multiple of row $k_m$ to a (non-zero) multiple of row $k_{m-1}$ to create a new row $k_{m-1}$ where the element in the $k_{m-1}$st column is replaced by a scalar quantity of choice. We choose the scalar to be
\begin{equation}\label{uf3}
M^{[i]}_{k_{m-1},k_{m-1}}=-\frac{h}{2}A_{i,k_{m-1}}.
\end{equation}
We repeat the argument. From \eqref{uf3}, and from \cite[Definition 3]{KQM}, as there is a path from $i$ to $k_{m-1}$, it follows that
\begin{equation}\label{uf0}
M^{[i]}_{i,k_{m-1}}=M^{[i]}_{k_l,k_{m-1}} \text{ for all } l=1,\ldots,m-1.
\end{equation}
Considering the $k_{m-2}$nd and the $k_{m-1}$st row of $\M^{[i]}$, as either $(k_{m-1},k_{m-2})$ or $(k_{m-2},k_{m-1})$ is an edge in $T$, we have
\begin{equation}\label{uf4}
M^{[i]}_{k_{m-2},l}=M^{[i]}_{k_{m-1},l} \text{ for all } l \neq k_{m-2},k_{m-1}.
\end{equation}
Because of \eqref{uf0}, equation \eqref{uf4} also holds for $l=k_{m-1}$, and thus the rows differ only in the $k_{m-2}$nd column. We can now add a (non-zero) multiple of row $k_m$ to a (non-zero) multiple of row to create a new row $k_{m-2}$ where the element in the $k_{m-2}$nd column is choosen to be
\[
M^{[i]}_{k_{m-2},k_{m-2}}=-\frac{h}{2}A_{i,k_{m-2}}.
\]
We continue making these elementary row-operations until we arrive at a matrix where
\begin{equation}\label{uf5}
M^{[i]}_{k_{2},k_{2}}=-\frac{h}{2}A_{i,k_2}.
\end{equation}
Now as $(i,k_2)$ or $(k_2,i)$ is an edge in $T$, we have
\begin{equation}\label{uf6}
M^{[i]}_{i,l}=M^{[i]}_{k_2,l} \text{ for all } l \neq i,k_2.
\end{equation}
Due to \eqref{uf5}, equation \eqref{uf6} also holds for $l=k_2$. But as the $i$th column of $\M^{[i]}$ equals ${\bf 1}$, equation \eqref{uf6} also holds for $l=i$. The rows are equal, and hence the determinant vanishes.
\end{proof}
\begin{example}
For the bushy tree on 4 vertices the Kahan discretisation satisfies $\M\x'=\x$ with
\begin{equation} \label{Mex}
\M=\begin{pmatrix}1 & 0 & 0 & 0 \\
 0 & 1 & 0 & 0 \\
 0 & 0 & 1 & 0 \\
 0 & 0 & 0 & 1 
 \end{pmatrix}
-\frac{h}{2}\begin{pmatrix}
C_1 + a_{1} x_{1} & b_{1} x_{1} & b_{2} x_{1} & b_{3} x_{1} \\
 c_{1} x_{2} & C_2 + a_{2} x_{2} & b_{2} x_{2} & b_{3} x_{2} \\
 c_{1} x_{3} & c_{2} x_{3} & C_3 + a_{3} x_{3} & b_{3} x_{3} \\
 c_{1} x_{4} & c_{3} x_{4} & b_{2} x_{4} & C_4 + a_{4} x_{4}
\end{pmatrix},
\end{equation}
where $C_i=(\A\x)_i$ is the cofactor of $x_i$. In terms of the functions
\begin{equation} \label{Ks}
\begin{split}
K_{2,1}(=K_{3,1}=K_{4,1})&=1+\frac{h}{2}(C_1+x_1(a_1-c_1)),\\
K_{1,2}&=1+\frac{h}{2}(C_2+x_2(a_2-b_1)),\\
K_{3,2}&=1+\frac{h}{2}(C_2+x_2(a_2-c_2)),\\
K_{4,2}&=1+\frac{h}{2}(C_2+x_2(a_2-c_3)),\\
K_{1,3}(=K_{2,3}=K_{4,3})&=1+\frac{h}{2}(C_3+x_3(a_3-b_2)),\\
K_{1,4}(=K_{2,4}=K_{3,4})&=1+\frac{h}{2}(C_4+x_4(a_4-b_3)),
\end{split}
\end{equation}
the Kahan map is explicitly given by
\begin{equation} \label{kmp}
\begin{pmatrix} x_1 \\ x_2 \\ x_3 \\ x_4 \end{pmatrix}'=\frac{1}{|\M|}\begin{pmatrix}
x_1K_{1,2}K_{1,3}K_{1,4} \\
x_2K_{2,1}K_{1,3}K_{1,4} \\
x_3K_{2,1}K_{3,2}K_{1,4} \\
x_4K_{2,1}K_{4,2}K_{1,3}
\end{pmatrix}.
\end{equation}
\end{example}

\begin{proposition} \label{prop2}
The cofactor of the Darboux polynomial $P_i$, which corresponds to the $i$-th edge $e_i=(u_i,v_i)$ (see Eq. \eqref{Pi}), is explicitly given by
\begin{equation*}% \label{DPK}
L_i\frac{\prod_{j\neq u_i,v_i} K_{u_i,j}}{|\M|}
\end{equation*}
where
\begin{equation} \label{Li}
L_{i}=1-\frac{h}{2}\left((\A.\x)_{u_i}-(A_{u_i,u_i}-A_{v_i,u_i})x_{u_i}\right),
\end{equation}
which is symmetric under $u_i\leftrightarrow v_i$.
\end{proposition}
\begin{proof}
In order to not have to carry many indices, we fix $i$ and denote the $i$-th edge by $e_i=(u,v)$. The $i$-th DP is then given by
\begin{equation*}
P_i=(c_i-a_{u})x_u+(a_{v}-b_i)x_v.
\end{equation*}
We note that due to the property of matrix $\A$ \cite[Corollary 8]{KQM}, that $j\neq u,v \implies A_{u,j}=A_{v,j}$, we have $K_{u,j}=K_{v,j}$. Using Proposition \ref{expl}, we find
\begin{align*}
P^\prime_i&=(c_i-a_{u})x^\prime_u+(a_{v}-b_i)x^\prime_v\\
&=\left((c_i-a_{u}) x_u \prod_{j\neq u} K_{u,j}+(a_{v}-b_i) x_v \prod_{j\neq v} K_{v,j}\right)|M|^{-1}\\
&=F \left(\prod_{j\neq u,v} K_{u,j}\right) |M|^{-1},
\end{align*}
with $F=(c_i-a_{u}) x_u K_{u,v}+(a_{v}-b_i) x_v  K_{v,u}$.
The prefactor is, using the fact that $P_i=(\A.\x)_v-(\A.\x)_u$, equal to
\begin{align*}
F&=
(c_i-a_{u}) x_u \left(
1-\frac{h}{2}\left((\A.\x)_v+(a_{v}-b_i)x_v\right)\right)\\
&\ \ \ +(a_{v}-b_i) x_v \left(
1-\frac{h}{2}\left((\A.\x)_u+(a_{u}-c_i)x_u\right)
\right)\\
&=
(c_i-a_{u}) x_u \left(
1-\frac{h}{2}\left((\A.\x)_u+(c_i-a_{u})x_u+2(a_{v}-b_i)x_v-Z(a_{v}-b_i)x_v\right)\right)\\
&\ \ \ +(a_{v}-b_i) x_v \left(
1-\frac{h}{2}\left((\A.\x)_u+(a_{u}-c_i)x_u+Z(c_i-a_{u})x_u\right)
\right)\\
&=
\left((c_i-a_{u}) x_u + (a_{v}-b_i) x_v \right) \left(
1-\frac{h}{2}\left((\A.\x)_u-(a_{u}-c_i)x_u\right)\right)\\
&=P_iL_i.
\end{align*}
which holds for all $Z$ and in particular for $Z=2$.
\end{proof}
\begin{example}
In terms of the matrix \eqref{Mex}, the functions \eqref{Ks} and
\begin{equation*} %\label{Ls}
\begin{split}
L_1=1+\frac{h}{2}(C_1-x_1(a_1-c_1))&=1+\frac{h}{2}(C_2-x_2(a_2-b_1))\\
L_2=1+\frac{h}{2}(C_2-x_2(a_2-c_2))&=1+\frac{h}{2}(C_3-x_3(a_3-b_2))\\
L_3=1+\frac{h}{2}(C_2-x_2(a_2-c_3))&=1+\frac{h}{2}(C_4-x_4(a_4-b_3)),
\end{split}
\end{equation*}
the cofactors of the preserved DPs \eqref{P123}, with respect to the Kahan map \eqref{kmp}, are given by
\[
\frac{L_1K_{1,3}K_{1,4}}{|\M|},\qquad
\frac{L_2K_{2,1}K_{1,4}}{|\M|},\qquad
\frac{L_3K_{2,1}K_{1,3}}{|\M|}.
\]
\end{example}

The following lemma will be used to establish an explicit expression for the Jacobian determinant for the Kahan map \eqref{mp}.
\begin{lemma} \label{lem}
The determinant of
\begin{equation*}
\Q=\d{\o}+\frac{h}{2}\left(\d{\x}.\A-\d{\A.\d{\x}.\o}\right).
\end{equation*}
is equal to
\[
|\Q|=\prod_{i=1}^{n-1} L_i.
\]
\end{lemma}
\begin{proof}
By the same argument as in the proof of Proposition \ref{expl}, if, for some constant $c$, we have \[
|\Q|=c\prod_{i=1}^{N-1} L_i,
\]
then $c=1$. Furthermore, as a linear combination of the columns $\q_i$ of $\Q-\d{\o}$, \[
\sum_ix_i\q_i={\bf 0},
\]
vanishes, the degree of $|\Q|$ in $h$ is 3. Therefore, it suffices to prove that
\[
L_i=0\implies |\Q|=0,
\]
for $i=1,\ldots, n-1$. Let $e_i=(u,v)$ again. We will show that the $u$th row and the $v$th row are dependent when $L_i=0$. For all $j\neq u,v$ we have
\[
x_vQ_{u,j}=x_v\frac{h}{2}x_uA_{u,j}=x_u\frac{h}{2}x_vA_{v,j}=x_uQ_{v,j}.
\]
When $j=u$ we have
\[
x_vQ_{u,u}\mid_{L_i=0}=x_v\left(1-\frac{h}{2}\left(\A.\x-A_{u,u}x_u\right)\right)\mid_{L_i=0}
=x_v\left(\frac{h}{2}x_uA_{v,u}\right)=x_uQ_{v,u}.
\]
and, by interchanging $u,v$ in the above, when $j=v$ we have
\[
x_vQ_{u,v}=x_uQ_{v,v}\mid_{L_i=0}.
\]
\end{proof}
\begin{example}
The matrix $\Q$ does not depend on the parameters $a_i$. For our running example we have
\[
\Q=\begin{pmatrix}
1 & 0 & 0 & 0 \\
0 & 1 & 0 & 0 \\
0 & 0 & 1 & 0 \\
0 & 0 & 0 & 1 
 \end{pmatrix} +\frac{h}{2}
\begin{pmatrix}
D_1 &   x_{1} b_{1} &   x_{1} b_{2} &   x_{1} b_{3} 
\\
   x_{2} c_{1} & D_2 &   x_{2} b_{2} & x_{2} b_{3} 
\\
   x_{3} c_{1} &   x_{3} c_{2} & D_3 & x_{3} b_{3} 
\\
   x_{4} c_{1} &   x_{4} c_{3} &   x_{4} b_{2} & D_4
\end{pmatrix},
\]
where
\begin{align*}
D_1&=-b_{1} x_{2}-b_{2} x_{3}-b_{3} x_{4},\quad
D_2=-b_{2} x_{3}-b_{3} x_{4}-c_{1} x_{1},\\
D_3&=-b_{3} x_{4}-c_{1} x_{1}-c_{2} x_{2},\quad
D_4=-b_{2} x_{3}-c_{1} x_{1}-c_{3} x_{2}.
\end{align*}
\end{example}
\begin{proposition}
The Jacobian determinant for the Kahan map \eqref{mp} is
\begin{equation} \label{J}
J=\frac{\left(\prod_{i=1}^{n-1} L_i\right)\left(\prod_{i=1}^n \left(\prod_{j=1}^n K_{i,j}\right)/K_{i,i} \right)}{|\M|^{n+1}}.
\end{equation}
\end{proposition}
\begin{proof}
Let us differentiate the equation $\M\x^\prime=\x$. Denoting differentiation w.r.t. $x_k$ by ${}_{;k}$, the components satisfy (using Kronecker's delta and summation convention)
\[
M_{i,j}x^\prime_{j;k}+M_{i,j;k}x^\prime_{j}=\delta_{i,k}.
\]
Rearranging and taking the determinant we find that the Jacobian determinant is given by
\[
J=\frac{|\d{\o}-\X|}{|M|}, \text{ where } X_{i,k}= M_{i,j;k} x^\prime_{j}.
\]
We create a matrix $\Y$ by dividing, for $i=1,\ldots,n$, the $i$th row of $\d{\o}-\X$ by $x^\prime_i/x_i$. Then, using \eqref{mp},
\[
|\d{\o}-\X|=\frac{\prod_{i=1}^n \left(\prod_{j=1}^n K_{i,j}\right)/K_{i,i} }{|\M|^{n}} |\Y|.
\]
If we can show $\Y=\Q$, then, by Lemma \ref{mp}, equation \eqref{J} follows. We have
\[
M_{i,j;k}=
\begin{cases}
-\frac{h}{2}\delta_{i,k}A_{i,j} & i\neq j \\
-hA_{i,i} & i=j=k \\
-\frac{h}{2} A_{i,k} & i=j\neq k,
\end{cases}
\]
and hence
\[
Y_{i,k}=\left(\delta_{i,k}-M_{i,j;k}x^\prime_{j}\right)x_i/x_i^\prime=
\begin{cases}
\frac{h}{2}A_{i,j}x_i & i\neq k\\
\left(1+\frac{h}{2}\left((\A.\x^\prime)_i+A_{i,i}x_i^\prime\right) \right)x_i/x_i^\prime & i=k.
\end{cases}
\]
Now consider the $i$th component of $\M\x^\prime=\x$. With
\[
M_{i,j}=\begin{cases}
-\frac{h}{2} x_iA_{i,j} & i \neq j \\
1-\frac{h}{2}\left((\A.\x)_i+A_{i,i}x_i\right) & i=j
\end{cases}
\]
we find
\[
x_i=M_{i,j}x_j^\prime=-\frac{h}{2}x_i\sum_{j\neq i} A_{i,j}x_j^\prime + \left(1-\frac{h}{2}\left((\A.\x)_i+A_{i,i}x_i\right)\right)x_i^\prime
\]
which implies
\[
\left(1+\frac{h}{2}\left((\A.\x^\prime)_i+A_{i,i}x_i^\prime\right) \right)x_i/x_i^\prime
=1-\frac{h}{2}\left((\A.\x)_i-A_{i,i}x_i\right),
\]
and hence $\Y=\Q$.
\end{proof}
\begin{example}
The Jacobian determinant of the map \eqref{kmp} is, in terms of \eqref{Kij}, \eqref{Li}, and \eqref{Mex},
\begin{equation} \label{Jex}
J=\frac{L_1L_2L_3K_{1,2}K_{3,2}K_{4,2}(K_{2,1}K_{1,3}K_{1,4})^3}{|\M|^5}.
\end{equation}
\end{example}
\begin{theorem}
The expression \eqref{d} is a rational Darboux function of the Kahan map \eqref{mp} with cofactor $J$ given by \eqref{J}.
\end{theorem}
\begin{proof}
As the cofactor of a product is the product of the cofactors we find, due to $\sum_{i=1}^n m_i=2(n-1)$, and using $e_i=(u_i,v_i)$ for $i=1,\ldots,n-1$,
\begin{align*}
d^\prime&=d\prod_{i=1}^n \left(\frac{\prod_{j\neq i}K_{i,j}}{|\M|} \right)^{2-m_i} \prod_{i=1}^{n-1} L_i\frac{\prod_{j\neq u_i,v_i} K_{u_i,j}}{|\M|}\\
&=d\frac{H\prod_{i=1}^{n-1} L_i}{|\M|^{n+1}},
\end{align*}
with
\begin{equation}\label{H}
H=\left(\prod_{i=1}^n \prod_{j\neq i} K_{i,j}^{2-m_i}\right)\left( \prod_{i=1}^{n-1} \prod_{j\neq u_i,v_i} K_{u_i,j}\right).
\end{equation}
As in the proof of Proposition 2, by $I^i$ we denote the index-set for which
\[
j\in I^i \Leftrightarrow \exists k\ e_k=(i,j) \text{ or } e_k=(j,i),
\]
so that the number of elements in $I^i$ equals $m_i$. And once again, for each vertex $i$ we view $T$ as a union of $m_i$ trees $T=\cup_{j\in I^i} T^i_j$ which are connected at $i$. Recall that $z(i,j)$ is the number of edges in $T^i_j$. We claim that
\begin{align}
H&=\left(\prod_{i=1}^{n-1}K_{u_i,v_i}K_{v_i,u_i}\right)
\left(\prod_{i=1}^{n-1}K_{u_i,v_i}^{z(u_i,v_i)-1}K_{v_i,u_i}^{z(v_i,u_i)-1}\right)\label{c1}\\
&=\prod_{i=1}^{n-1}K_{u_i,v_i}^{z(u_i,v_i)}K_{v_i,u_i}^{z(v_i,u_i)}\label{c2}\\
&=\prod_{i=1}^n \left(\prod_{j=1}^n K_{i,j}\right)/K_{i,i}\label{c3},
\end{align}
which would show $J$ is the cofactor of $d$.

\cite[Definition 3]{KQM} states that the weight of edge $(i,j)\in T$ equals the weight of $(k,j)$ or $(j,k)$ with $i\in T^j_k$, and \cite[Proposition 7]{KQM} states that if $(u,v)$ is an edge in $T$, then for all $w\neq u,v$ the edges $(u,w)$, $(v,w)$ carry the same weight. It is easy to see that the converse is also true, i.e., if the edges $(u,w)$, $(v,w)$ carry the same weight, for all edges $(u,v)\in T$ and $w\neq u,v$, then the weight of edge $(i,j)\in T$ equals the weight of $(k,j)$ or $(j,k)$ with $i\in T^j_k$. Those properties carry over to the entries of the adjacency matrix $A_{i,j}$, cf. \cite[Corollary 8]{KQM}, and to the functions $K_{i,j}$ as defined by \eqref{Kij}, cf. the proof of Proposition \ref{prop2}. Consider the expression \eqref{c3}. Let $(k,j)$ or $(j,k)$ be an edge in $T$, then $k\in I^j$. Since $i\in T^j_k\setminus \{j\}$, iff $K_{i,j}=K_{k,j}$ the degree of $K_{k,j}$ in \eqref{c3} is $z(k,j)$ and hence \eqref{c3} equals \eqref{c2}.
Next, consider the first factor of \eqref{H}. Because
\[
\sum_{\underset{i\neq j}{i\in T^j_k}}\Big( 2-m_i \Big) = 1,
\]
we have
\begin{align*}
\prod_{i=1}^n \prod_{j\neq i} K_{i,j}^{2-m_i}&=\prod_{j=1}^n \prod_{i\neq j} K_{i,j}^{2-m_i}
=\prod_{j=1}^n \prod_{k\in I^j} \prod_{\underset{i\neq j}{i\in T^j_k}} K_{i,j}^{2-m_i}
=\prod_{j=1}^n \prod_{k\in I^j} K_{k,j}\\
&=\prod_{i=1}^{n-1} K_{u_i,v_i}K_{v_i,u_i},
\end{align*}
which is the first factor of \eqref{c1}. Finally, the second factor of \eqref{H} is
\begin{align*}
\prod_{i=1}^{n-1} \prod_{j\neq u_i,v_i} K_{u_i,j}
&=\prod_{j=1}^{n} \prod_{\underset{j\neq u_i,v_i}{i=1}}^{n-1} K_{u_i,j}
=\prod_{j=1}^{n} \prod_{k\in I^j} \prod_{\underset{i\neq j,k}{i\in T^j_k}} K_{i,j}
=\prod_{j=1}^{n} \prod_{k\in I^j} K_{k,j}^{z(k,j)-1}\\
&=\prod_{i=1}^{n-1}K_{u_i,v_i}^{z(u_i,v_i)-1}K_{v_i,u_i}^{z(v_i,u_i)-1},
\end{align*}
which is the second factor of \eqref{c1}.
\end{proof}
\begin{example}
The expression \eqref{d} is a Darboux function of the map \eqref{kmp} which has cofactor $J$, the Jacobian determinant \eqref{Jex}.
\end{example}

\section{Integrals for Kahan maps of LV-systems on graphs}
A class of homogeneous Lotka--Volterra systems is associated with any graph $G$ on $n$ vertices; when $G$ contains both a tree with $n$ vertices and a cycle of length 3 or greater, we call such a system a {\em (Lotka--Volterra) $G$-system}. Each edge of the graph is associated with a DP (preserved under Kahan discretisation), and each subgraph of $G$ that is a tree on $n$ vertices is associated with an invariant measure (preserved under Kahan discretisation). A ratio of two invariant measures is a first integral, as illustrated in the following example.

\begin{example}
Consider the 4D Lotka--Volterra system with matrix
\[
\A=\begin{pmatrix} a_{{1}}&b_{{1}}&b_{{2}}&b_{{3}} \\
c_{{1}}&a_{{2}}&b_{{2}}&b_{{3}}\\
c_{{1}}&c_{{2}}&a_{{3}}&b_{{3}}\\
c_{{1}}&c_{{2}}&b_{{2}}&a_{{4}}
\end{pmatrix}.
\]
obtained from \eqref{Mex} by taking $c_3=c_2$. The system admits four additional DPs
\begin{align*}
P_{{1}}=(c_{{1}}-a_{{1}})x_{{1}}+(a_{{2}}-b_{{1}})x_{{2}}\qquad &P_{{2}}=(c_{{2}}-a_{{2}})x_{{2}}+(a_{{3}}-b_{{2}})x_{{3}}\\
P_{{3}}=(c_{{2}}-a_{{2}})x_{{2}}+(a_{{4}}-b_{{3}})x_{{4}}\qquad
&P_{{4}}=(b_{{2}}-a_{{3}})x_{{3}}+(a_{{4}}-b_{{3}})x_{{4}},
\end{align*}
one for each edge in the graph of Figure \ref{Gfv}.
\begin{figure}[h]
\begin{center}
\scalebox{.7}{\begin{tikzpicture}
    \node[shape=circle,draw=black,line width=0.5mm] (1) at (-2,2) {1};
	\node[shape=circle,draw=black,line width=0.5mm] (2) at (0,0) {2};
	\node[shape=circle,draw=black,line width=0.5mm] (3) at (2,2) {3};
	\node[shape=circle,draw=black,line width=0.5mm] (4) at (4,0) {4};
	\path [-,line width=0.5mm, ] (1) edge node[above] {} (2);
    \path [-,line width=0.5mm, ] (2) edge node[above] {} (3);
    \path [-,line width=0.5mm, ] (3) edge node[above] {} (4);
    \path [-,line width=0.5mm, ] (2) edge node[above] {} (4);
\end{tikzpicture}}
\caption{\label{Gfv} Graph on $4$ vertices.}
\end{center}
\end{figure}
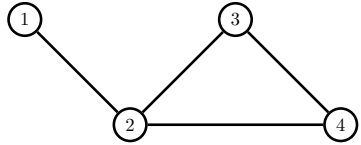
We identify three subgraphs, as in Figure \ref{tsg}.
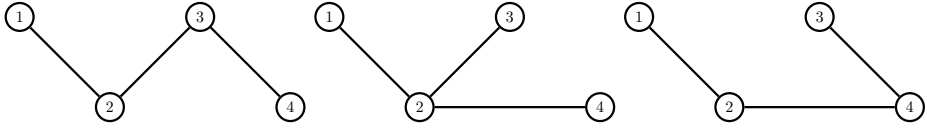
\begin{figure}[h]
\begin{center}
\scalebox{.6}{\begin{tikzpicture}
    \node[shape=circle,draw=black,line width=0.5mm] (1) at (-2,2) {1};
	\node[shape=circle,draw=black,line width=0.5mm] (2) at (0,0) {2};
	\node[shape=circle,draw=black,line width=0.5mm] (3) at (2,2) {3};
	\node[shape=circle,draw=black,line width=0.5mm] (4) at (4,0) {4};
	\path [-,line width=0.5mm, ] (1) edge node[above] {} (2);
    \path [-,line width=0.5mm, ] (2) edge node[above] {} (3);
    \path [-,line width=0.5mm, ] (3) edge node[above] {} (4);
%    \path [-,line width=0.5mm, ] (2) edge node[above] {} (4);
\end{tikzpicture}}
\scalebox{.6}{\begin{tikzpicture}
    \node[shape=circle,draw=black,line width=0.5mm] (1) at (-2,2) {1};
	\node[shape=circle,draw=black,line width=0.5mm] (2) at (0,0) {2};
	\node[shape=circle,draw=black,line width=0.5mm] (3) at (2,2) {3};
	\node[shape=circle,draw=black,line width=0.5mm] (4) at (4,0) {4};
	\path [-,line width=0.5mm, ] (1) edge node[above] {} (2);
    \path [-,line width=0.5mm, ] (2) edge node[above] {} (3);
%    \path [-,line width=0.5mm, ] (3) edge node[above] {} (4);
    \path [-,line width=0.5mm, ] (2) edge node[above] {} (4);
\end{tikzpicture}}
\scalebox{.6}{\begin{tikzpicture}
    \node[shape=circle,draw=black,line width=0.5mm] (1) at (-2,2) {1};
	\node[shape=circle,draw=black,line width=0.5mm] (2) at (0,0) {2};
	\node[shape=circle,draw=black,line width=0.5mm] (3) at (2,2) {3};
	\node[shape=circle,draw=black,line width=0.5mm] (4) at (4,0) {4};
	\path [-,line width=0.5mm, ] (1) edge node[above] {} (2);
%    \path [-,line width=0.5mm, ] (2) edge node[above] {} (3);
    \path [-,line width=0.5mm, ] (3) edge node[above] {} (4);
    \path [-,line width=0.5mm, ] (2) edge node[above] {} (4);
\end{tikzpicture}}\caption{\label{tsg} Three subgraphs that are trees.}
\end{center}
\end{figure}
Each tree in Figure \ref{tsg} comes with a measure, and these have the following densities
\[
d_{{1}}=x_{{1}}x_{{4}} P_1P_2P_4,\qquad
d_{{2}}=\frac{x_{{1}}x_{{3}}x_{{4}} P_1P_2P_3}{x_{{2}}},\qquad
d_{{3}}=x_{{1}}x_{{3}} P_1P_3P_4.
\]
Taking ratios $K_1=d_1/d_2$ and $K_2=d_1/d_3$ we find the following integrals
\[
K_{{1}}=\frac{x_{{2}} P_4}{x_{{3}} P_3},\qquad K_{{2}}=\frac{x_{{4}} P_2}{x_{{3}} P_3 },
\]
for the special case $c_3=c_2$ of the Kahan map \eqref{kmp}.
%\[
%\begin{pmatrix}
%x_1\\
%x_2\\
%x_3\\
%x_4
%\end{pmatrix}
%\mapsto
%\frac{1}{|\M|}
%\begin{pmatrix}
%x_{{1}}K_{1,2} K_{1,3} K_{1,4}\\
%x_{{2}}K_{2,1} K_{2,3} K_{2,4}\\
%x_{{3}}K_{3,1} K_{3,2} K_{3,4}\\
%x_{{4}}K_{4,1} K_{4,2} K_{4,3}
%\end{pmatrix}
%\]
%where $\M$ and $K_{i,j}$ are given by \eqref{M} and \eqref{Kij} respectively.
We note that these integrals are not independent, they satisfy
$(a_2 - c_2)K_1 + (a_4 - b_3)K_2 = a_3 - b_2$.
\end{example}

Not every graph gives rise to a unique class of $G$-systems.  

\begin{proposition} For a graph $G$ which contains a cycle, let $G'$ be the graph obtained from $G$ by adding edges between any pair of vertices in the cycle. Every $G$-system is a $G'$-system.
\end{proposition}
\begin{proof}
Let the cycle have length $\ell$ with edges (without loss of generality) 
$(1,2),(2,3),$ $\dots,(\ell-1,\ell),(\ell,1)$. The following equations are satisfied:
\begin{align*}
A_{t,k}&=A_{t+1,k},\ \forall\ 1\leq t < \ell,\ k\neq t,t+1,\\
A_{\ell,k}&=A_{1,k},\ \forall\ k\neq \ell,1.
\end{align*}
We have to prove
\[
A_{i,k}=A_{j,k},\ \forall\ 1\leq i\neq j \leq \ell, k\neq i,j.
\]
For all $k\neq i,j$, one of the paths $i,i+1,\ldots,j$ or $i,i-1,\ldots,j$ (where indices are taken modulo $\ell$) does not contain $k$.
In the first case we have $A_{i,k}=A_{i+1,k}=\cdots=A_{j,k}$, and otherwise $A_{i,k}=A_{i-1,k}=\cdots=A_{j,k}$.
\end{proof}

Therefore the graphs in which we are interested are those for which restriction to any cycle yields a complete graph. (See Figures \ref{fig:graph1}, \ref{fig:graph2} and \ref{fig:graph3} for the graphs on $n=4,5,6$ vertices.) These are the so-called block-graphs, in which every biconnected component is complete \cite{harary}.

The Kahan map of a $G$-system has an invariant measure corresponding to each subgraph which is an $n$-tree. There can be many such subgraphs, as many as $\frac{1}{2}n!$ (for the complete graph on $n$ nodes). The ratio of any two such invariant measures is an integral of the Kahan map.  

\begin{proposition}
Let $G$ be a graph on $n$ vertices containing a complete subgraph of size $\ell\ge 3$. The Kahan map of
a $G$-system has at least $\ell-2$ functionally independent first integrals.
\end{proposition}
\begin{proof}
Consider the subgraph of $G$ consisting of a cycle contained in the complete subgraph together with a number of trees attached to its vertices, such that the subgraph contains $n$ edges. Deleting any edge in the cycle yields a tree and an associated invariant measure of the Kahan map.
Let $(i,j,k)$ be three adjacent vertices in the cycle. The integral given by the ratio of the densities \eqref{dd} corresponding to the two trees given by (i) deleting edge $\alpha :=(i,j)$ and (ii) deleting edge $\beta :=(j,k)$, is
$$ \frac{x_i P_\beta}{x_k P_{\alpha}},$$
the factors of \eqref{dd} corresponding to all other edges and vertices cancelling.

To show functional independence, consider without loss of generality the cycle with edges $\{(1,2),\dots,(\ell-1,\ell),(\ell,1)\}$ with 
edges labelled $\{1,\dots,\ell\}$ respectively. The $\ell-2$ integrals
\[
\frac{x_1 P_{2}}{x_{3}P_1},\frac{x_2 P_{3}}{x_{4}P_2},\ldots,\frac{x_{\ell-2}P_{\ell-1}}{x_{\ell}P_{\ell-2}},
\]
where the $i$th integral is a rational function of $x_i$, $x_{i+1}$, and $x_{i+2}$, are functionally independent because the Jacobian matrix of these functions is upper triangular.
\end{proof}

In dimension 4, 5, and 6 there are 2, 6, and 16 classes of $G$-systems, respectively. The graphs associated with each class are shown in Figures \ref{fig:graph1}, \ref{fig:graph2} and \ref{fig:graph3}, respectively.

\begin{proposition}
The  number of functionally independent first integrals of the Kahan map of a $G$-system is at least
\begin{equation}\label{sellimt}
\sum_i (\ell_i-2),
\end{equation}
where $\ell_1,\ell_2,\dots$ are the sizes of the complete subgraphs of $G$.
\end{proposition}
\begin{proof}
This proposition is easy to prove after one realises that the integral
\begin{align*}
\frac{x_1 P_2}{x_{3}P_1}&=\frac{x_1(\alpha_2 x_2+\beta_2 x_3)}{x_{3}(\alpha_1 x_1+\beta_1 x_2)}\\&=\frac{x_1/x_2(\alpha_2+\beta_2 x_3/x_2)}{x_{3}/x_2(\alpha_1 x_1/x_2+\beta_1)}
\end{align*}
only depends on 2 variables $x_1/x_2=y_1$ and $x_2/x_3=y_2$, the so called edge variables, cf. \cite[Section 5]{KQM}. The integrals corresponding to one complete subgraph do not depend on the edge-variables of another complete subgraph. Hence the integrals and edge-variables can be ordered so that the 
Jacobian matrix is upper triangular.
\end{proof}

For each graph $G$ in Figures \ref{fig:graph1}, \ref{fig:graph2} and \ref{fig:graph3}, the ODE of the  $G$-system is superintegrable, while its Kahan discretisation is measure-preserving with at least one integral. The integer attached to each graph is the number of functionally independent integrals \eqref{sellimt}.

\begin{figure}[h]
\begin{center}
\includegraphics[height=2.2cm]{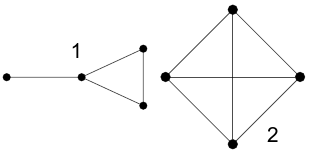}
\end{center}
\caption{\label{fig:graph1} The 2 graphs on 4 vertices associated with distinct classes of Lotka--Volterra $G$-systems.}
\end{figure}
\begin{figure}[h]
\begin{center}
\includegraphics[height=4.6cm]{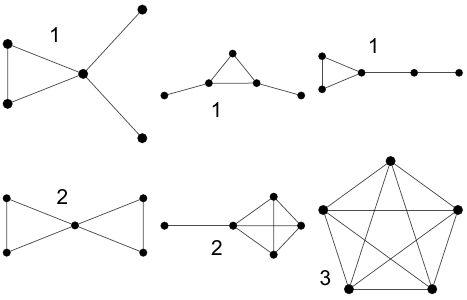}
\end{center}
\caption{\label{fig:graph2} The 6 graphs on 5 vertices associated with distinct classes of Lotka--Volterra $G$-systems.} 
\end{figure}
\begin{figure}[h]
\begin{center}
\includegraphics[height=9.5cm]{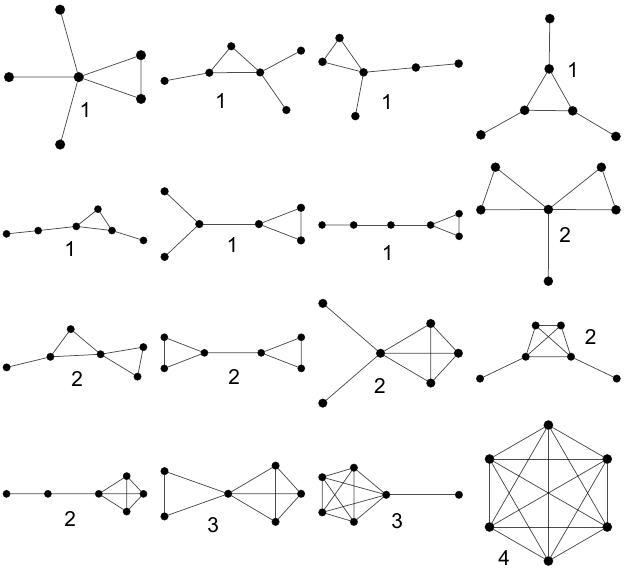}
\end{center}
\caption{\label{fig:graph3} The 16 graphs on 6 vertices associated with distinct classes of Lotka--Volterra $G$-systems.
The integer attached to each graph is the number of functionally independent integrals.}
\end{figure}

We have shown that the Kahan map for Lotka--Volterra $G$-systems in dimension $n$, where $G$ is the complete graph on $n$ vertices, is measure-preserving and has at least $n-2$ functionally independent integrals. This is not enough for integrability. If it had $n-1$ functionally independent integrals, it would be superintegrable, and this could in principle be detected via the method of algebraic entropy or degree growth \cite{kamp}. We implemented this method for $n=4$ and observed an exponential rate of degree growth, indicating that the map does not have any integrals additional to those given above (nor any other properties which would make it integrable, such as a symplectic structure with respect to which sufficiently many integrals are in involution).

\medskip
\noindent{\bf 
Acknowledgment.} This paper is dedicated to Hans Munthe-Kaas and Brynjulf Owren on the occasion of their 120th birthday. We are grateful for their friendship and collegiality over many years and in many countries. RQ, RM, and DM have particularly fond memories of their very fruitful and enjoyable visits to Norway on multiple occasions.

\medskip
% The information below will be filled in by AIMS production staff.
Received xxxx 20xx; revised xxxx 20xx; early access xxxx 20xx.
\medskip

\end{document}